\documentclass[a4paper,11pt]{amsart}
\usepackage{amsthm, amssymb, latexsym, amsmath, color, mathrsfs, caption, mathtools}
\usepackage[utf8]{inputenc}

\usepackage{enumitem}

\usepackage{graphicx,color}
\usepackage[dvipsnames]{xcolor}
\usepackage{mathrsfs}
\usepackage{pdfpages}
\usepackage{tabularx}

\usepackage{enumitem}

\usepackage{tikz-cd}

\usepackage{verbatim}
\usepackage{fancyvrb}
\usepackage{xcolor}
\usepackage{listings}

\definecolor{codegreen}{rgb}{0, 0.6, 0}
\definecolor{codegray}{rgb}{0.5, 0.5, 0.5}
\definecolor{codepurple}{rgb}{0.58, 0, 0.82}
\definecolor{backcolour}{rgb}{0.95, 0.95, 0.92}

\lstdefinelanguage{Macaulay2}
{
xleftmargin=.2in, 
xrightmargin=.2in, 
basicstyle={\footnotesize\ttfamily}, 
keywordstyle={\color{blue}}, 
commentstyle={\color{codegreen}}, 
stringstyle={\color{red!40!black}}, 
rulecolor=\color{yellow}, 
basewidth={1.2ex}, 
sensitive=false, 
morecomment=[l]{--}, 
morecomment=[s]{-*}{*-}, 
morestring=[b]", 
escapechar={`}, 
escapebegin={\rmfamily}, 
morekeywords={Sequence, QQ, saturate, load, random, degree, genus, topComponents, ideal, Ext, minors, quotient, intersect, map, kernel, preimage, codim, sheaf, matrix, hilbertPolynomial, Projective, false, sheafExt, ann, cooker, flatten, gens, entries, basis, apply}
}
\lstalias{Macaulay2output}{Macaulay2}

\usepackage{geometry}
\usepackage{marginnote}
\usepackage{appendix}

\usepackage{calligra}

\usepackage{color}
\usepackage[normalem]{ulem}

\usepackage{array}
\usepackage{makecell}
\usepackage{tabularx}
\usepackage[framemethod=tikz]{mdframed}

\newcommand{\extp}{{\textstyle \bigwedge}}
\newcommand{\nwc}{\newcommand}

\newcommand{\p}[1]{{\mathbb{P}^{#1}}}
\newcommand{\opn}{{\mathcal O}_{\mathbb{P}^{n}}}
\newcommand{\pn}{{\mathbb{P}^{n}}}
\newcommand{\tp}[1]{{\rm T}{\mathbb{P}^{#1}}}
\newcommand{\tpn}{{\rm T}{\mathbb{P}^{n}}}
\newcommand{\om}[2]{\Omega_{\mathbb{P}^{#2}}^{#1}}
\newcommand{\omn}[1]{\Omega_{\mathbb{P}^{n}}^{#1}}

\newcommand{\tD}{{\rm T}_\mathscr{D}}
\newcommand{\nD}{{\rm N}_\mathscr{D}}

\newcommand{\sD}{\mathscr{D}}
\nwc{\C}{\mathbb{C}}

\nwc{\K}{ k }

\nwc{\dd}{\mathcal{D}}

\DeclareMathOperator{\codim}{{codim}}

\DeclareMathOperator{\Tor}{Tor}
\DeclareMathOperator{\sing}{Sing}

\DeclareMathOperator{\rad}{rad}

\newcommand{\PP}{\mathbb{P}}

\newcommand{\sI}{\mathscr{I}}

\input{xypic.tex} \xyoption{all}

\usepackage{float}
\setcellgapes{3pt}

\newtheorem{coro}{Corollary}[section]

\newtheorem{lema}[coro]{Lemma}
\newtheorem{prop}[coro]{Proposition}

\theoremstyle{definition}
\newtheorem{exem}[coro]{Example}
\newtheorem{rem}[coro]{Remark}

\newtheorem{innerthmintro}{Theorem}
\newenvironment{thmintro}[1]
{\renewcommand\theinnerthmintro{#1}\innerthmintro}
{\endinnerthmintro}

\usepackage[utf8]{inputenc}

\usepackage{listings}
\usepackage{algorithm}
\usepackage{algpseudocode}

\usepackage{hyperref}
\hypersetup{
	colorlinks=true,
	linkcolor=blue,
	filecolor=magenta, 
	urlcolor=cyan,
}
\makeatletter
\@namedef{subjclassname@2020}{%
 \textup{2020} Mathematics Subject Classification}
\makeatother

\author{Alan Muniz}
\address{Universidade Estadual de Campinas (UNICAMP) \\ Instituto de Matemática, Estatística e Computação Científica (IMECC) \\ Departamento de Matem\'atica \\
Rua S\'ergio Buarque de Holanda, 651\\ 13083-970 Campinas-SP, Brazil}
\curraddr{Departamento de Matem\'atica \\ Centro de Ci\^encias Exatas e da Natureza \\ Universidade Federal de Pernambuco \\ Recife - PE, CEP 50740-560, Brazil }
\email{alan.nmuniz@ufpe.br}

\subjclass[2020]{14D20, 14J60, 14F06, 13D02}

\keywords{Differential forms, Distributions, Singular scheme, Vanishing Locus, Syzygy}

\date{June, 2024}

\title{{\it p}-forms from syzygies}

\begin{document}

\begin{abstract}
These notes aim to develop a tool for constructing polynomial differential $p$-forms vanishing on prescribed loci through syzygies of homogeneous ideals. Examples are provided through implementing this method in Macaulay2, particularly examples of instanton bundles of charges $4$ and $5$ on $\p3$ that arise in this construction. 
\end{abstract}

\maketitle
\section{Introduction}

In Mathematics, being able to compute explicit examples is very important. In particular, when studying distributions on projective spaces, such examples are governed by homogeneous polynomial differential forms which induce exact sequences of the form
\[
0 \longrightarrow F \longrightarrow \tp{n} \stackrel{\omega}{\longrightarrow} N \longrightarrow 0 
\]
where $F$ is a rank $n-p$ reflexive sheaf and $N$ is a rank $p$ torsion-free sheaf. Properties of the vanishing locus of the $p$-form $\omega$ reflect on properties of the sheaves $F$ and $N$; thus providing the $p$-form leads to the understanding of the sheaves. In this direction, we prove the following result. Fix $k$ a field.

\begin{thmintro}{A} \label{mainthm}
Let $Z\subset \PP^n$ a closed subscheme with (saturated) homogeneous ideal $I_Z \subset R = k[x_0, \dots, x_n]$ and let $\mathcal{A}^p(Z)$ be the $R$-module of polynomial differential $p$-forms vanishing on $Z$. Then we have an exact sequence of graded $R$-modules:
\[
0 \longrightarrow I_Z \otimes  \iota_{\rm rad}\bigwedge^{p+1}V^* 
\longrightarrow \mathcal{A}^p(Z) \longrightarrow   \Tor_p^R(I_Z,\K)(p+1) \longrightarrow 0,
\]
where $V^* = \langle dx_0, \dots , dx_n \rangle$ and $\rad$ is the radial vector field (see \eqref{eq:rad}). Moreover, if $(I_Z)_d = 0$ then $\mathcal{A}^p(Z)_d \cong \Tor_p^R(I_Z,\K)_{d+p+1}$. 
\end{thmintro}

The $R$-module $\Tor_p^R(I_Z,\K)$ can be identified with the space of $p$-th syzygies of a minimal set of generators for $I_Z$. Given a (minimal) free resolution 
\[
0 \longrightarrow F_n \stackrel{\phi_n}{\longrightarrow} F_{n-1} \stackrel{\phi_{n-1}}{\longrightarrow} \dots \longrightarrow F_0  \stackrel{\phi_0}{\longrightarrow} I_Z \longrightarrow 0,
\]
we construct a map $\xi_{p} \colon \Tor_p^R(I_Z,\K) \rightarrow \mathcal{A}^p(Z)$ given by differentiating  and combining the entries of $\phi_1, \dots ,\phi_p$, i.e., $\xi_{p}$ depends on syzygies up to order $p$, see Proposition \ref{prop:splitting}.

The use of syzygies to describe distributions goes back at least to the work of Campillo and Olivares \cite{CO:Polarity}, see also \cite{CO:CB} and references therein. In \cite[\S 4]{CJM} the case of $1$-forms is essentially described, serving as a prelude to the present work. Note that for $p=1$ Theorem \ref{mainthm} gives a slightly more complete version of \cite[Proposition 4.5]{CJM}. 

After recalling some relevant concepts in \S 2, we prove in \S 3 Proposition \ref{prop:syzkforms0}, from which Theorem \ref{mainthm} follows. Finally, we provide some examples in \S4. In Example \ref{ex:ch4} we give an example of an instanton bundle of charge $4$ on $\p3$ which is, up to twist, the conormal sheaf of a foliation by curves singular along $5$ disjoint lines; this construction was first observed in \cite{ACM-inst}, though without explicitly referring to foliations. In Example \ref{ch5}, we apply the same construction to produce an instanton of charge $5$, from a foliation by curves singular along a disjoint union of two double lines of genus $-3$, cf. \cite[\S6]{CJM:curves}. Our computations come from implementing these routines in Macaulay2 \cite{M2}. These are compiled in the ancillary file \verb|syz-k-forms.m2|, available at \url{https://github.com/alannmuniz/syz-k-forms.git}

\subsection*{Acknowledgements} This work is supported by INCTmat/MCT/Brazil, CNPq grant number 160934/2022-2, and FAEPEX-Unicamp 2797/23. 
It originates from discussions with Maurício Corrêa, Marcos Jardim, and Renato Vidal, to whom I am grateful. Thanks to an anonymous referee for many important suggestions.

\section{Preliminaries and notation}
We begin by recalling some basic facts and establishing the notation used throughout the paper. Let $\K$ be a field, that we may assume is algebraically closed of characteristic zero. Fix $V$ a $\K$-vector space, $n := \dim V - 1$,  and let $\pn = \PP(V)$ the projective space of lines in $V$ through the origin. Let $\tpn$ and $\Omega^1_\pn$ denote the tangent and cotangent bundles of $\pn$. We have the Euler sequence,
\begin{equation}\label{seq:euler}
    0 \longrightarrow \opn \stackrel{\rad}{\longrightarrow} \opn(1) \otimes V \longrightarrow \tpn \longrightarrow 0.
\end{equation}
Then we may identify $V = H^0(\tpn(-1))$ as the space of constant vector fields (on the affine space $\mathbb{A}^{n+1}_{\K}$). Fixing homogeneous coordinates $(x_0: \dots: x_{n})$, we have that $V$ is spanned by the derivations $\frac{\partial}{\partial x_0}, \dots , \frac{\partial}{\partial x_n}$ so that the map $\rad$ in \eqref{seq:euler} is written as the inclusion of the radial vector field:
\begin{equation}\label{eq:rad}
    \rad = x_0 \frac{\partial}{\partial x_0} + \cdots + x_n \frac{\partial}{\partial x_n},
\end{equation}
which is sometimes called the Euler derivation in the literature. Dualizing \eqref{seq:euler} we get
\begin{equation}\label{seq:eulerdual}
    0 \longrightarrow \omn1 \longrightarrow \opn(-1) \otimes V^\ast \stackrel{\iota_{\rad}}{\longrightarrow} \opn \longrightarrow 0,
\end{equation}
and considering $\{dx_0, \dots, dx_n\}$ the basis of $V^\ast$ dual to $\{\frac{\partial}{\partial x_0}, \dots , \frac{\partial}{\partial x_n}\}$ we see that $\iota_{\rad}$ is the contraction -- or interior product -- of (local) differential $1$-forms  with the radial vector field: $\iota_{\rad}(\omega) = \omega(\rad)$. Furthermore, we take exterior powers of \eqref{seq:eulerdual} to arrive at
\begin{equation}\label{seq:kforms}
    0 \longrightarrow \omn{p} \longrightarrow \opn(-p) \otimes \extp^p V^\ast \stackrel{\iota_{\rad}}{\longrightarrow} \omn{p-1} \longrightarrow 0,
\end{equation}
where $\iota_{\rad}\omega(v_1, \dots, v_{p-1}) = \omega(\rad, v_1, \dots, v_{p-1})$ is the contraction of $p$-forms with $\rad$. 

Note that from \eqref{seq:kforms} we have that global sections of $\omn{p}(d+p+1)$ are in bijection with homogeneous differential $p$-forms
\[
\omega = \sum_{0\leq i_1< \dots < i_p \leq n} A_{i_1\dots i_p } dx_{i_1} \wedge \dots\wedge dx_{i_p } 
\]
satisfying $\iota_{\rad} \omega = 0$ and $\deg A_{i_1\dots i_p } = d+1$.

To fix notation, let $R$ denote the polynomial ring $R = \K[x_0, \dots, x_n]$ and let $\Omega^p_{R} = R \otimes \extp^p V^\ast$ be the free $R$-module of polynomial differential $p$-forms; $\Omega^0_{R} = R$ and $\Omega^{l}_{R} = 0$ for $l<0$.  Then the radial vector field $\rad$ defines a $R$-linear map $\iota_{\rad} \colon \Omega^p_{R} \rightarrow \Omega^{p-1}_{R}$ so that its kernel is $\bigoplus_{r} H^0(\omn{p}(r))$.


\subsection{Distributions} 
Given $1\leq p \leq n-1$, a codimension $p$ distribution $\sD$ on $\pn$ is defined by a short exact sequence
\begin{equation}\label{seq:distkpn}
    0 \longrightarrow \tD \stackrel{\phi}{\longrightarrow} \tpn \stackrel{\psi}{\longrightarrow} \nD \longrightarrow 0 
\end{equation}
such that $\nD$ is a rank $p$ torsion-free sheaf; hence $\tD$ is a rank $n-p$ reflexive sheaf. The distribution $\sD$
is integrable, i.e., a foliation, if $\phi(\tD)$ is closed under the Lie bracket of vector fields. 

Taking exterior powers of \eqref{seq:distkpn} yields a differential $p$-form $\omega \in H^0(\omn{p}(d +p+1))$, where $d := c_1(\tD(-1))$ is called the degree of $\sD$. The map $\psi$ is given by contraction: $\psi(v) = \iota_v\omega$. The coefficients of $\omega$ generate the singular scheme $\sing(\sD)\subset \pn$, supported on the set of points where $\nD$ is not free. As $\nD$ is torsion-free, $\codim \sing(\sD) \geq 2$.

Therefore, to study degree-$d$ codimension-$p$ distributions on $\pn$ we may focus on homogeneous $p$-forms representing global sections of $\omn{p}(d +p+1)$. But first, notice that not every such $p$-form induces a distribution. 

\begin{exem}\label{ex:nonLDS}
Consider the $2$-form $\omega \in H^0(\omn{2}(3))$ given by
\[
\omega = x_0(dx_1\wedge dx_2 + dx_3\wedge dx_4) - dx_0\wedge(x_1dx_2 - x_2dx_1 + x_3dx_4-x_4dx_3).
\]
One can readily check that $\omega$ defines a ``trivial distribution'', $\tD = 0$. Indeed, the contraction map $\iota_\bullet \omega \colon \tpn \rightarrow \Omega_\pn^{p-1}(3)$ is injective. For instance, on the affine chart $U_0 = \{x_0 =1 \}$ we have natural local coordinates $(x_1, x_2, x_3, x_4)$ and, for any local vector field $v = \sum_{j=1}^4 a_j\frac{\partial}{\partial x_j}$, we have
\[
\iota_v\omega|_{U_0} = \iota_v(dx_1\wedge dx_2 + dx_3\wedge dx_4) = a_1dx_2 - a_2dx_1 + a_4dx_3-a_3dx_4 
\]
which vanishes if and only if so does $v$.
\end{exem}

Fortunately, there is a computable characterization for locally decomposable and integrable forms. 

\begin{rem}
To simplify our notation, we set $\extp^0 V^\ast = \extp^0 V = \K$ and $\iota_v\omega := \omega$ for $v\in \extp^0 V$.
\end{rem}

\begin{lema}\label{def:LDS}
A homogeneous $p$-form $\omega$ on $\pn$ is locally decomposable off the singular set (LDS) if 
\[
(\iota_v \omega) \wedge \omega = 0 , \quad \text{ for every } v\in \extp^{p-1}V;
\]
here $\iota_{v_1\wedge \dots \wedge v_{p-1}} \omega := \iota_{v_{p-1}} \cdots \iota_{v_{1}}\omega$. Moreover, a LDS form $\omega$ is integrable if 
\[
(\iota_v \omega) \wedge d\omega = 0 , \quad \text{ for every } v\in \extp^{p-1}V.
\]
\end{lema}

The proof is an iterated application of de Rham-Saito Division Lemma \cite{Saito} after localizing to the principal open subset $D_+(f)\subset \pn$, for $f$ a coefficient of $\omega$, and we leave it to the reader.

Given an LDS $p$-form $\omega$ defining a distribution $\sD$, we want to compute its tangent and normal sheaves. To do so, we analyze a suitable complex of sheaves associated with $\omega$. This was observed in \cite[p.13]{CCJ} for codimension $p=1$ and the general case is similar. Taking exterior powers of the Euler sequence \eqref{seq:eulerdual} we get a natural inclusion $\Omega_\pn^{p-1}(d+p+1) \hookrightarrow \opn(d+2)\otimes \extp^{p-1}V^\ast$. On the other hand, as in \eqref{seq:euler}, $\tpn$ is the cokernel of $\rad \colon \opn \rightarrow \opn(1)\otimes V$, induced by the radial vector field. Hence, we consider the composition 
\[
C_\omega\colon \opn(1)\otimes V \twoheadrightarrow \tpn \stackrel{\omega}{\longrightarrow} \Omega_\pn^{p-1}(d+p+1)  \hookrightarrow \opn(d+2)\otimes \extp^{p-1}V^\ast.
\]
Note that $\nD$ is isomorphic to the image of $C_\omega$ and we also get:
\begin{equation}\label{eq:complex1}
    \opn \stackrel{\rad }{\longrightarrow} \opn(1)\otimes V \stackrel{C_\omega}{\longrightarrow} \opn(d+2)\otimes \extp^{p-1}V^\ast.
\end{equation}
This complex is interesting because the associated complex of free $R$-modules is computationally convenient to describe $\tD$. The following is straightforward.

\begin{lema}\label{lem:tangcoh}
Let $\sD$ be a codimension-$p$ distribution on $\pn$ of degree $d$ given by a homogeneous $p$-form $\omega$. Then $\tD$ is the middle cohomology of the complex \eqref{eq:complex1} and $\nD$ is the image of $C_\omega$.
\end{lema}





\section{Forms with prescribed vanishing locus}

Now we turn to the main objective of this work, which is to describe the module of homogeneous $p$-forms, not necessarily LDS, vanishing along some given subscheme. To describe distributions, one may further apply Lemma \ref{def:LDS}.   

Let $Z\subset \pn$ be a closed subscheme with ideal sheaf $\sI_Z$ and consider 
\[
\mathcal{A}^p(Z) := \bigoplus_{d\geq 0}H^0(\Omega_\pn^p(d+p+1)\otimes \sI_Z).
\]
the $R$-module of twisted differential $p$-forms that vanish on $Z$. Let $I_Z$ denote the saturated homogeneous ideal of $Z$, i.e., $I_Z := \bigoplus_j H^0(\sI_Z(j)) \subset R$.


\begin{prop}\label{prop:syzkforms0}
Let $Z \subset \mathbb{P}^n$ be a closed subscheme then
\begin{equation}\label{seq:diffmodule}
    0 \longrightarrow I_Z \otimes \iota_{\rad} \Omega^{p+1}_{R} \longrightarrow \mathcal{A}^p(Z) \longrightarrow   \Tor_p ^R(I_Z,\K)(p+1) \longrightarrow 0.
\end{equation}
\end{prop}

\begin{proof}
Consider the $(p+1)$-st exterior power of Euler sequence tensored with the sheaf $\sI_Z(d+p+1)$:
\[
0 \longrightarrow \Omega^{p+1}_{\mathbb{P}^n}(d+p+1) \otimes \sI_Z \longrightarrow \sI_Z(d)\otimes \extp^{p+1}V^\ast \stackrel{\iota_{\rad}}{\longrightarrow} \Omega^p_{\mathbb{P}^n}(d+p+1) \otimes \sI_Z \longrightarrow 0,
\]
which is exact since $\Omega^p_{\mathbb{P}^n}$ is locally free. From the long sequence of cohomology, we get
\[
H^0(\sI_Z(d))\otimes \extp^{p+1}V^\ast \stackrel{\iota_{\rad}}{\longrightarrow} H^0( \Omega^p_{\mathbb{P}^n}(d+p+1) \otimes \sI_Z) \stackrel{\phi}{\longrightarrow} H^1( \Omega^{p+1}_{\mathbb{P}^n}(d+p+1) \otimes \sI_Z) 
\]
is exact. From \cite[Theorem 5.8]{EI:SYZ5} the image of $\phi$ is precisely $\Tor^R_p (I_Z, \K)_{d+p+1}$. Taking the direct sum over $d\geq 0$, we get the desired sequence of $R$-modules. Note that the image of $\iota_{\rad}$ after the sum is isomorphic to $I_Z \otimes \iota_{\rad} \Omega^{p+1}_{R}$.
\end{proof}

In most cases of interest, one wants to describe degree $d$ distributions singular along a $Z$ such that $H^0(\sI_Z(d))= 0$ so that
\[
H^0( \Omega^p_{\mathbb{P}^n}(d+p+1) \otimes \sI_Z) \simeq \Tor^R_p (I_Z, \K)_{d+p+1}.
\]
It is an interesting open question to decide whether $H^0(\sI_Z(d))= 0$ holds for $Z = \sing(\sD)$. This is true, for instance, if $k=1$ and $\dim Z = 0$, see \cite[Lemma 4.2]{CJM}. 

\subsection{\texorpdfstring{$p$}{p}-forms and syzygies}

Note that the sequence of graded $\K$-vector spaces underlying \eqref{seq:diffmodule} must split, and we derive such a splitting from the syzygies of $I_Z$. Consider the minimal graded free resolution 
\begin{equation}\label{seq:resideal}
    0 \longrightarrow F_n \stackrel{\phi_n}{\longrightarrow} F_{n-1} \stackrel{\phi_{n-1}}{\longrightarrow} \dots \longrightarrow F_0  \stackrel{\phi_0}{\longrightarrow} I_Z \longrightarrow 0,
\end{equation}
where $F_i = \bigoplus_j R(-j)^{\beta_{i,j}}$. Recall that, since the resolution is minimal,
\begin{equation}\label{eq:tor}
    \Tor^R_p (I_Z, \K) \simeq F_p \otimes \K = \bigoplus_j \K^{\beta_{p,j}},
\end{equation}
with $\K^{\beta_{p,j}}$ in degree $j$. Moreover, fixed the minimal generators given by $\phi_0$, the module of $p$-th syzygies of $I_Z$ is the image of $\phi_p $. Note that if we tensor \eqref{seq:resideal} with the free module $\Omega^{l}_{R}$ we get an exact sequence 
\begin{equation}
    0 \longrightarrow F_n\otimes \Omega^{l}_{R} \stackrel{\phi_n}{\longrightarrow} F_{n-1}\otimes \Omega^{l}_{R} \stackrel{\phi_{n-1}}{\longrightarrow} \dots \longrightarrow F_0 \otimes \Omega^{l}_{R} \stackrel{\phi_0}{\longrightarrow} I_Z\otimes \Omega^{l}_{R} \longrightarrow 0
\end{equation}
where $\phi_j = \phi_j \otimes 1$ by abuse of notation. We then define a $\K$-linear map $\delta \colon \Omega^p_{R} \rightarrow \Omega^{p+1}_{R}$ by setting 
\[
\delta \omega = \frac{d\omega}{\deg \omega},
\]
on a homogeneous $\omega$. Here $\deg \omega$ is the total degree of $\omega$ considering $\deg dx_i = \deg x_i = 1$. The important property $\delta$ has is that 
\[
\iota_{\rad} \delta \omega = \omega, \quad \text{for} \quad \omega \in \ker \iota_{\rad}.
\]
Given a matrix of $p$-forms $G = (g_{ij})$ we denote $\delta G = (\delta g_{ij})$ and similarly for $\iota_{\rad}$; we use the dot $\cdot$ to denote matrix multiplication, whether the entries are commutative or not. 

To construct a $1$-form that vanishes on $Z$, we take $t \in \Tor_1^R(I_Z,\K) \cong \bigoplus_j \K^{\beta_{1,j}}$, which we regard as a column vector of elements of $k$ with the appropriate grading. The matrix $\phi_0$ is a row vector of (minimal) generators of $I_Z$ and the columns of $\phi_1$ are the first syzygies; in particular, $\phi_1 t$ is a first syzygy. Then we apply $\delta$ and multiply the matrices: $\xi_1(t):= \phi_0\cdot \delta \phi_1 \cdot t$. It vanishes on $Z$ since the coefficients belong to $I_Z$, and it may descend to the projective space since 
\[
\iota_{\rad} \xi_1(t) = \iota_{\rad} (\phi_0\cdot \delta \phi_1 \cdot t) = \phi_0\cdot \phi_1 \cdot t = 0
\]
by the above relation; by convention $\iota_{\rad}F= 0$ for any polynomial $F$. Note, however, that $\xi_1(t)$ is only homogeneous if $t\in \Tor_1^R(I_Z,\K)$ is homogeneous, i.e., it has nonzero entries in only one degree.
For $2$-forms, the procedure is similar, take $t\in \Tor_2^R(I_Z,\K)$ and define $\xi_2(t) := \phi_0\cdot \delta( \phi_1 \cdot \delta\phi_2) \cdot t$. Notice that we differentiate the matrix $\phi_2$, multiply the result by $\phi_1$, then differentiate the product. Following the same strategy, we construct $p$-forms vanishing on $Z$ with the following proposition. Recall that one can also produce $p$-forms from $I_Z \otimes \iota_{\rad}\Omega^{p+1}_{R}$, i.e., as a combination $\eta = \sum F_i\eta_i$ with $F_i\in I_Z$ and $\iota_{\rad}\eta_i = 0$. The forms we obtain via the above procedure are not of this type. 

\begin{prop}\label{prop:splitting}
The $k$-linear morphism $\xi_{p} \colon \Tor_p^R(I_Z,\K) \rightarrow \mathcal{A}^p(Z)$ defined by 
\[
\xi_{p}(t) =  (\phi_0 \circ \delta  \circ \phi_1 \circ \cdots \circ  \delta  \circ \phi_p ) \cdot t,
\]
(alternating $\delta$ and multiplication by $\phi_i$) is injective, and its image does not intersect the image of $I_Z \otimes \iota_{\rad}\Omega^{p+1}_{R}$. 
\end{prop}

\begin{proof}
First, note that, due to $R$-linearity, $\iota_{\rad} (\phi_j \cdot \delta \phi_{j+1}) = \phi_j \cdot \phi_{j+1} = 0$. It is then straightforward to show that $\iota_{\rad} \xi_{p}(t) = 0$.

We will assume by contradiction that $\xi_{p}(t)$ belongs to the image of $I_Z \otimes \iota_{\rad}\Omega^{p}_{R}$ and conclude that $t=0$, hence proving both claims at once. From the assumption there exists $\eta_0$ a column vector of $(p+1)$-forms such that
\[
 \xi_{p}(t) = \phi_0 \cdot \iota_{\rad}\eta_0. 
\]
Hence 
\[
\phi_0 \cdot \left( \delta (\phi_1  \circ \cdots \circ  \delta  \circ \phi_p)\cdot  t - \iota_{\rad}\eta_0 \right) = 0
\]
and, due to the exactness of \eqref{seq:resideal} twisted by $\Omega^p_{R}$, there exists $\eta_1$ a column vector of $p$-forms such that 
\[
\delta (\phi_1  \circ \cdots \circ  \delta  \circ \phi_p)\cdot  t - \iota_{\rad} \eta_0 = \phi_1 \cdot \eta_1.
\]
Applying $\iota_{\rad}$ we get, due to $R$-linearity,
\[
 \phi_1 \cdot \delta  (\phi_2 \circ  \cdots \circ  \delta  \circ \phi_p)\cdot  t  = \phi_1 \cdot \iota_{\rad} \eta_1 \Longrightarrow  \phi_1 \cdot \left( \delta  (\phi_2 \circ  \cdots \circ  \delta  \circ \phi_p)\cdot  t   - \iota_{\rad} \eta_1 \right) = 0.
\]
Thus there exists $\eta_2$, a column vector of $(p-1)$-forms, such that $\delta (\cdots \delta \phi_p ))\cdot  t  - \iota_{\rad} \eta_1 = \phi_2 \cdot \eta_2$. Iterating this process we arrive at
\[
\delta \phi_p  \cdot t - \iota_{\rad} \eta_{p-1}  = \phi_p  \cdot \eta_p  
\]
where $\eta_p $ is a column vector of $1$-forms. Hence, there exists a matrix of polynomials $A$ such that 
\[
t - \iota_{\rad}\eta_p   = \phi_{p+1} \cdot A. 
\]
Since the resolution \eqref{seq:resideal} is minimal, each entry of $\phi_{p+1}$ is a homogeneous polynomial; the same is true for $\iota_{\rad}\eta_p $. On the other hand, the entries of $t\in \Tor_p ^R(I_Z,\K)$ are constants. Thus, comparing degrees, we see that $t=0$.

\end{proof}

Note that from \eqref{seq:diffmodule} we expect that $f\xi_{p}(t) \in I_Z \otimes \iota_{\rad}\Omega^{p}_{R}$, for any homogeneous polynomial $f \in R$. Indeed, we can write
\[
f\xi_{p}(t) =  \phi_0 \cdot \iota_{\rad}\left(\delta f\wedge  \delta (\phi_1 \cdot \delta (\cdots \delta \phi_p ))\cdot  t\right).
\]
Also we have that $\xi_{p}(t)$ is not homogeneous unless $t\in \Tor^R_p (I_Z, \K)_m$ for some $m$; in this case the total degree of $\xi_{p}(t)$ is $m$. Moreover, if $h^0(\sI_Z(d)) = 0$ we can pass to the first linear strand of \eqref{seq:resideal}: 
\[
 0 \longrightarrow F_n^0 \stackrel{\phi_n^0}{\longrightarrow} F_{n-1}^0 \stackrel{\phi_{n-1}^0}{\longrightarrow} \dots \longrightarrow F_0^0  \stackrel{\phi_0^0}{\longrightarrow} I_Z 
\]
where $F_j^0 = R(-j-d-1)^{\beta_{j,j+d+1}}$ and $\phi_j^0$ are the corresponding linear blocks. Thus, $H^0( \Omega^p_{\mathbb{P}^n}(d+p+1) \otimes \sI_Z) \simeq \Tor^R_p (I_Z, \K)_{d+p+1}$ may be computed from
\[
\xi_p^0(t) = \phi_0^0 \cdot d\phi_1^0 \wedge \dots \wedge d\phi_p ^0 \cdot t,
\]
which involves only the degree $d+1$ generators of $I_Z$. 


\section{Examples}
In this section, we compute some examples. We will focus on degree $d$ distributions singular along $Z$ such that $(I_Z)_d = 0$. This is expected to always hold for $Z$ the full singular scheme, see the introduction to \cite{GJM}.

\begin{exem}[$n= 2,p=1, d=1 $]
 Let us start with a simple example. Let $Z \subset \PP^2$ be a reduced subscheme of length $3$ not contained in a line. Then we may suppose 
\[
I_Z = (x_0,x_1) \cap (x_0,x_2) \cap (x_1,x_2) = (x_0x_1,x_0x_2,x_1x_2).
\]
The resolution is given by 
\[
0 \longrightarrow R^2 \xrightarrow{
\begin{pmatrix}
x_2 & 0 \\ -x_1 & x_1 \\ 0 & -x_0  
\end{pmatrix} 
} R^3 \xrightarrow{
\begin{pmatrix}
x_0x_1 & x_0x_2 & x_1x_2
\end{pmatrix}
} I_Z \longrightarrow 0
\]
Thus $H^0(\Omega_{\PP^2}^1(3)\otimes \sI_Z)\cong \K^2$ spanned by 
\[
\omega_1 = x_0x_1dx_2 - x_0x_2dx_1 \quad \text{and} \quad \omega_2 = x_0x_2dx_1 - x_1x_2dx_0.
\]
Note that both $\omega_1$ and $\omega_2$ vanish along a line and a point, but a general linear combination of them vanishes precisely at $Z$.    
\end{exem}

\begin{exem}[$n= 3,p=1, d=1 $] Also for $d=1$ consider $Z\subset \PP^3$ a twisted cubic:
\[
0 \longrightarrow R^2 \xrightarrow{
\begin{pmatrix}
x_0 & x_1 \\ x_1 & x_2 \\ x_2 & x_3
\end{pmatrix} 
} R^3 \xrightarrow{
\begin{pmatrix}
x_1x_3-x_2^2 & x_1x_2-x_0x_3 & x_0x_2-x_1^2
\end{pmatrix}
} I_Z \longrightarrow 0
\]
Then we also get $H^0(\Omega_{\PP^2}^1(3)\otimes \sI_Z)\cong \K^2$. A general element vanishes only on $Z$.
\end{exem}

\begin{exem}[$n= 3,p=1, d=1 $] Next, we describe a pathological example for codimension one and degree one on $\PP^3$. Consider $Z$ given by  
$$
I_Z=(x_0^2, x_1^2, x_0x_2, x_1 x_2, x_2^2 -x_0 x_1).
$$ 
It is a $0$-dimensional scheme of length of $5$ supported on a single point. It is a simple example of a point that is not a local complete intersection. Thus, there is no local $1$-form vanishing only on $Z$. Due to Theorem \ref{mainthm}, any $1$-form vanishing on $Z$ can be written as $\omega = Adx_0 + Bdx_1 +Cdx_2$ where
\begin{align*}
	A & = t_0x_1x_2 +t_2x_1^2 +t_3x_0x_2 +t_4(x_0x_1-x_2^2) \\
	B & = -t_0x_0x_2+t_1x_1x_2-t_2(x_0x_1-x_2^2)-t_4x_0^2 \\
	C & = -t_1x_1^2-t_2x_1x_2-t_3x_0^2 +t_4x_0x_2
\end{align*}
and $t_0, \dots, t_4 \in \mathbb{C}$. Note that $\omega$ does not depend on $x_3$ so it is a linear pullback of a $1$-form $\eta$ on $\mathbb{P}^2$ and the singular locus is thus a cone over the singular locus of $\eta$. Therefore, any $\omega$ vanishing on $Z$ must vanish along $3$ lines concurring at $Z_{\rm red}$.
\end{exem}

\begin{exem}[$n= 3,p=2, d=2 $]
Any $(n-1)$-form $\omega \in H^0(\Omega^{n-1}_{\PP^n}(d+n))$ can be written as $\omega = \iota_{\rad} \iota_v \, dx_0 \wedge \dots \wedge dx_n$ for some vector field $v \in H^0(\tpn(d-1))$; in particular, it is LDS. The distribution, actually the foliation, described by $\omega$ is often better described by $v$, and to get this vector field from $\omega$ we just note that $\iota_v \, dx_0 \wedge \dots \wedge dx_n = \frac{1}{d+n}d\omega$. If $v= \sum_{j=0}^n a_j\frac{\partial}{\partial x_j}$ then, the singular scheme is defined by the maximal minors of the matrix
\[
\begin{pmatrix}
    x_0 & \cdots & x_n \\ a_0 & \cdots & a_n
\end{pmatrix},
\]
which coincides with the ideal generated by the coefficients of $\omega$, up to saturation.

Now we specialize to $\PP^3$. In \cite{CJM:curves} foliations by curves on $\PP^3$ are studied with a special focus on those having locally free conormal sheaf $\nD^\vee$. Nonetheless, estimates on Chern classes predict a foliation of degree $2$ with conormal sheaf satisfying $c_1(\nD^\vee) = -5, c_2(\nD^\vee) = 9$ and $c_3(\nD^\vee) = 3$; here $c_3 > 0$ implies non-locally-free. Then  \cite[Theorem 4.1]{CJM:curves} translates it to predicting a foliation singular along $Z = C \cup P$ where $C$ is a curve of degree $2$ and genus $-2$, and $P$ is zero-dimensional of length $3$. Then consider, for instance,
\begin{align*}
    C & = V(x_0^2, x_0x_1, x_1^2, x_0(x_2^2-x_3^2) - x_1x_3x_2),\\
    P & = V(x_2-x_0, x_2+x_1,x_3) \cup V(x_1-x_0, x_2,x_3+x_0) \\ & \qquad \cup V(x_1-2x_0, x_2+x_0,x_3-x_0).
\end{align*}
Computing the syzygies we can construct, inside a space of dimension $4$, the vector field $v = \sum_{j=0}^3 a_j\frac{\partial}{\partial x_j}$ where 

\begin{align*}
    a_0 & = -4\,x_{0}^{2}-50\,x_{0}x_{1}+20\,x_{1}^{2},\\
    a_1 & = -40\,x_{0}^{2}+16\,x_{0}x_{1}-10\,x_{1}^{2},\\
    a_2 & = 40\,x_{0}^{2}-45\,x_{0}x_{1}+35\,x_{1}^{2}-4\,x_{0}x_{2}+50\,x_{1}x_{2}+30\,x_{0}x_{3},\\
    a_3 & = 50\,x_{0}^{2}+40\,x_{0}x_{1}-40\,x_{1}^{2}+30\,x_{0}x_{2}-4\,x_{0}x_{3}+20\,x_{1}x_{3}.\\
\end{align*}
To check that $v$ is singular precisely along $Z$, one may follow the Macaulay2 routine below.
\begin{lstlisting}[language=Macaulay2]
i1 : R =  QQ[x_0..x_3];
i2 : C = ideal(x_0^2, x_0*x_1, x_1^2, x_0*(x_2^2-x_3^2) - x_1*x_3*x_2);
o2 : Ideal of R
i3 : P = intersect(ideal(x_2-x_0, x_2+x_1,x_3),ideal(x_1-x_0, x_2,x_3+x_0),
     ideal(x_1-2*x_0, x_2+x_0,x_3-x_0));
o3 : Ideal of R
i4 : Z =  intersect(C,P);
o4 : Ideal of R
i5 : a0 = -4*x_0^2-50*x_0*x_1+20*x_1^2;
i6 : a1 = -40*x_0^2+16*x_0*x_1-10*x_1^2;
i7 : a2 = 40*x_0^2-45*x_0*x_1+35*x_1^2-4*x_0*x_2+50*x_1*x_2+30*x_0*x_3;
i8 : a3 = 50*x_0^2+40*x_0*x_1-40*x_1^2+30*x_0*x_2-4*x_0*x_3+20*x_1*x_3;
i9 : singD = saturate minors(2, matrix{{x_0,x_1,x_2,x_3},{a0,a1,a2,a3}});
o9 : Ideal of R
i10 : Z == singD
o10 = true
\end{lstlisting}
\end{exem}

\begin{exem}\label{ex:ch4}
In \cite{ACM-inst}, the authors provide a construction for instanton bundles $F$ on $\p3$ of charge $4$ as a twist of the kernel of a map $\om{1}{3}(1) \to\sI_{Z}(3)$, where $Z$ is the disjoint union of $5$ lines with no $5$-secant. In our notation, $F(-3) = \nD^\vee$ is the conormal sheaf of a degree-$3$ foliation by curves $\sD$. Next, we show how to provide explicit examples of such sheaves with the help of the ancillary file \verb|syz-k-forms.m2| (available at \url{https://github.com/alannmuniz/syz-k-forms.git}). 
\begin{lstlisting}[language=Macaulay2]
i1 : load "syz-k-forms.m2"
i2 : R = QQ[x_0..x_3];
i3 : C = dsLns(5,R); -- 5 random lines
o3 : Ideal of R
i4 : om = rOmg(2,3,C); -- random 2-form of degree 3 vanishing on C
                         1                  1
o4 : Matrix (R[dx ..dx ])  <-- (R[dx ..dx ])
                 0    3             0    3
i5 : C == sing om -- check C is the whole singular scheme
o5 = true
i6 : N = conSheaf om; -- compute the conormal sheaf
i7 : F = N(3);
i8 : chern F -- compute its Chern classes
o8 = (0, 4, 0)
o8 : Sequence
i9 : HH^1(F(-2))
o9 = 0
o9 : QQ-module
\end{lstlisting}
\end{exem}

\begin{exem}\label{ch5}
Similar to the previous example, if we set $C$ as the disjoint union of two double lines of genus $-3$, we get an instanton bundle $F$ of charge $5$.

\begin{lstlisting}[language=macaulay2]
i1 : load "syz-k-forms.m2"
i2 : R = QQ[x_0..x_3];
i3 : C1 = ideal(x_0^2, x_0*x_1, x_1^2, x_0*x_2^3 - x_1*x_3^3);
o3 : Ideal of R
i4 : C2 = ideal(x_2^2, x_2*x_3, x_3^2, x_2*x_0^3 - x_3*x_1^3);
o4 : Ideal of R
i5 : saturate(C1+C2) == R --check that they are disjoint
o5 = true
i6 : C = intersect(C1,C2);
o6 : Ideal of R
i7 : om = rOmg(2,3,C);
                         1                  1
o7 : Matrix (R[dx ..dx ])  <-- (R[dx ..dx ])
                 0    3             0    3
i8 : N = conSheaf om; -- compute the conormal sheaf
i9 : F = N(3);
i10 : chern F -- compute its Chern classes
o10 = (0, 5, 0)
o10 : Sequence
i11 : HH^1(F(-2))
o11 = 0
o11 : QQ-module
\end{lstlisting}
\end{exem}

 \bibliographystyle{abbrvurl}
 \bibliography{biblio}
\end{document}